\documentclass{amsart}
\usepackage[T1]{fontenc}
\usepackage{amssymb, amsmath, color, textcomp, url,tikz}
\usetikzlibrary{matrix,arrows}
\usetikzlibrary{fit}
\usetikzlibrary{positioning}
\usepackage[labelformat=empty]{caption}

\usepackage{soul}

\usepackage{mdwlist}

\RequirePackage{ifpdf}
\ifpdf
\usepackage[pdftex]{hyperref}
\else
\usepackage[hypertex]{hyperref}
\fi

\usepackage{enumerate,xspace}

\newtheorem*{teo}{Theorem}

\theoremstyle{plain}
\newtheorem{theorem}{Theorem}[section]
\newtheorem{cor}[theorem]{Corollary}
\newtheorem{prop}[theorem]{Proposition}
\newtheorem{lemma}[theorem]{Lemma}
\newtheorem*{claim}{Claim}
\newenvironment{claimproof}{\noindent\textit{Proof of
		Claim.}}{\hfill\qedsymbol \tiny{Claim}
	\medskip}

\theoremstyle{definition}
\newtheorem{remark}[theorem]{Remark}
\newtheorem{fact}[theorem]{Fact}
\newtheorem{definition}[theorem]{Definition}

\newtheorem*{notation}{Notation}

\newtheorem*{conj}{Conjecture}

\newcommand{\nc}{\newcommand}

\nc{\Z}{\mathbb{Z}}
\nc{\Q}{\mathbb{Q}}
\nc{\N}{\mathbb{N}}
\nc{\F}{\mathbb{F}}
\nc{\UU}{\mathbb{U}}
\nc{\C}{\mathbb{C}}

\nc{\M}{\mathcal{M}}
\nc\LL{\mathcal L}
\nc\II{\mathcal I}

\nc{\stt}{\operatorname{St}}
\nc{\stab}{\operatorname{Stab}}
\nc{\GO}[1]{G_{#1}^{00}}
\nc{\band}[1]{\bar d_{\mathcal{#1}}}
\nc\BD{\operatorname{BD}}

\nc{\dcl}{\operatorname{dcl}}
\nc{\dclq}{\operatorname{acl^\text{eq}}}

\nc{\acl}{\operatorname{acl}}
\nc{\aclq}{\operatorname{acl^\text{eq}}}
\nc{\nf}[1]{_{\mid {#1}}}
\nc{\restr}[1]{\xspace_{\upharpoonright {#1}}}
\nc{\sbgp}[1]{\langle\xspace {#1}\xspace\rangle}

\nc\CAN{\operatorname{CB}}
\nc\inv{ ^{-1}}

\nc{\tp}{\operatorname{tp}}
\nc\cb{\operatorname{Cb}}
\nc\U{\operatorname{U}}
\nc{\cf}{\text{cf.\,}}
\nc{\eg}{\text{e.g. }}

\def\Ind#1#2{#1\setbox0=\hbox{$#1x$}\kern\wd0\hbox to
	0pt{\hss$#1\mid$\hss} \lower.9\ht0\hbox to
	0pt{\hss$#1\smile$\hss}\kern\wd0}
\def\Notind#1#2{#1\setbox0=\hbox{$#1x$}\kern\wd0\hbox to
	0pt{\mathchardef\nn="0236\hss$#1\nn$\kern1.4\wd0\hss}\hbox to
	0pt{\hss$#1\mid$\hss}\lower.9\ht0 \hbox to
	0pt{\hss$#1\smile$\hss}\kern\wd0}

\def\indip{\mathop{\ \ \hbox to 0pt{\hss$\mid^{\hbox to
				0pt{$\scriptstyle P$\hss}}$\hss}
		\lower4pt\hbox to 0pt{\hss$\smile$\hss}\ \ }}
\def\nindip{\mathop{\ \ \hbox to 0pt{\hss$\!\not{\mid}^{\hbox to
				0pt{$\scriptstyle\, P$\hss}}$\hss}
		\lower4pt\hbox to 0pt{\hss$\smile$\hss}\ \ }}

\begin{document}

	\title{Stabilizers, Measures and IP-sets}
	\date{\today}
	
	\author{Amador Martin-Pizarro and Daniel Palac\'in}
	\address{Abteilung f\"ur Mathematische Logik, Mathematisches Institut,
		Albert-Ludwig-Universit\"at Freiburg, Ernst-Zermelo-Stra\ss e 1, D-79104
		Freiburg, Germany}
	 	\email{pizarro@math.uni-freiburg.de}

 \address{Departamento de \'Algebra, Geometr\'ia y Topolog\'ia; Facultad de Matem\'aticas;
	Universidad Complutense de Madrid; 28040 Madrid, Spain}

	\email{dpalacin@ucm.es}
	
	\thanks{Research supported by MTM2017-86777-P as well as by the Deutsche
		Forschungsgemeinschaft (DFG, German Research Foundation) - 
		Project number 2100310201 and 2100310301, part of the ANR-DFG 
		program GeoMod}
	
	\begin{abstract}
		The purpose of this article is to provide elementary model-theoretic proofs
		to some existing results on sumset phenomena and IP sets, motivated by
		Hrushovski's
		work on the stabilizer theorem.
	\end{abstract}
	
	\maketitle

\section*{Introduction}

A subset of the integers is called an \emph{IP set} (this stands for
\emph{infinite-dimensional parallele\-piped}) if it contains all
finite
sums of elements (without repetitions) of an infinite subset.  More
gene\-rally, a
subset $A$ of an infinite group $G$ (possibly non-abelian) is an
{\em IP set}  if there exists an enumeration $(x_i)_{i\in \mathbb
N}$ of an infinite subset $B$ of $A$  such that $A$ contains
\[
\mathrm{FP}(B)  = \{ x_{i_1}\cdots x_{i_k} \ |  \ i_1< \ldots< i_k
\}_{k\in \mathbb N}.
\]
Hindman's influential theorem \cite{nH74} states that for any
finite coloring
on the na\-tural numbers, there is a monochromatic IP set.
Hindman's original
proof was elementary, yet rather technical. Galvin and Glazer gave a beautiful
proof (see the account in \cite{wC77}) using the topology of the Stone space of
ultrafilters on $\N$ as well as Ellis-Numakura's lemma on the existence of idempotent
ultrafilters  \cite{kN52, rE58}. Every IP subset $A$ of $G$ has the \emph{productset
property}, that is, there are two infinite subsets
$B$ and $C$ of $G$ such that \[ B\cdot C=\{b\cdot c\ |\ b\in B, c\in C
\}\subseteq A,\]

\noindent which is sometimes referred to as the \emph{sumset property}
when the group $G$ is abelian.  Nathan\-son \cite{mN80}
attributes
to Erd\H{o}s the  so-called
\emph{$B+C$ conjecture}, which was stated by Erd\H{o}s and
Graham
in \cite{EG80} as an  \emph{old problem}:
\begin{conj}
Every subset $A$ of
$\mathbb N$ of positive lower density, that is,
\[
  \liminf\limits_{n\to \infty} \frac{\big|A \cap
    [1,n]\big|}{n}>0,\] \noindent must have the sumset property.
\end{conj}

A stronger version of the conjecture has
been recently positively answered in work of Moreira, Richter and Robertson \cite{MRR19} using ergodic
theory as well as
ultrafilters, which also played a fundamental role in the paper of Di Nasso et al. \cite{DGJLLM15}. In the
latter,  an instance of the above conjecture was solved for subsets $A$ of
Banach density \[ \BD(A)=\lim\limits_{n\to \infty}\sup\limits_{m \in \N}
\frac{|A \cap [m+1,m+n]|}{n}>\frac{1}{2}.\]

Di Nasso showed that Erd\H{o}s's $B+C$ conjecture can be phrased in
terms of \emph{independent} sums of ultrafilters (see \cite[Lemma
5.1]{MRR19}). This reformulation carries over to coheir sum of types
(that is, ultrafilters on the boolean algebra of definable sets), as
proved in  \cite[Theorem 3.4]{ACG19} by Andrews, Conant and Goldbring. For a group definable in a
stable theory, the principal generic type is indeed idempotent with
respect to the coheir (or independent) product.  Outside of stability,
there need not be a single principal generic type, yet in simple
theories a partial result was obtained by Pillay, Scanlon and Wagner in \cite[Proposition
2.2]{PSW98}: Given three principal generic (complete) types $p$, $q$ and $r$ over
an elementary substructure $M$, we have that there are independent
realizations $a$ of $p$ and $b$ of $q$ with $a\cdot b$ realizing $r$,
or equivalently, using the notation of section \ref{S:Eq}, the
type-definable set $r \ \cap (p\ast q)$ is not empty. Partial types $\Sigma$ with
$\Sigma \cap (\Sigma\inv\ast \Sigma)\neq\emptyset$ are not product-free (see Lemma \ref{L:prod_free}), a notion which was already considered in the finitary context by Gowers \cite{wG08} in order to give a negative
answer to a question of Babai and So\'s \cite{BS85} on the existence
of large product-free sets in finite groups. The second author
recently provided \cite{dP19} a model-theoretic proof of Gowers's
result, relating the existence of large product-free sets to
Hrushovski's \emph{stabilizer theorem} \cite{eH12}.

Hrushovski's stabilizer theorem allows one to generalise some of the
techniques from geometric stability theory to an arbitrary context,
replacing in a clever way the role of forking independence and
genericity by the ideal of null-sets with respect to a nonstandard
counting measure. A fundamental notion present in Hrushovski's
original proof is equationality, as introduced by Srour \cite{gS88} in
order to develop a local approach to noetherianity for the category of
definable sets in any structure. However, Hrushovski's proof does not
use the full strength of equationality. In section \ref{S:Eq}, we take
the opportunity to further explore the role of equationality in some
consequences of the stabilizer theorem, and relate it to a local
version for large product-free sets and type-definable connected components, proving the following result:

\begin{teo}[Corollary \ref{C:St-prod_free}]\label{T:MainTeo}
Let $G$ be a definable group with an $\emptyset$-type-definable  left-translation ideal $\II$ of definable subsets of $G$. Assume that $\II$ is indiscernibly prime (see Definition \ref{D:S1}). The following conditions are equivalent for every elementary submodel $M$:
\begin{itemize}
    \item If a partial type $\Sigma$ over $M$ is wide (\it{i.e.} avoids all definable sets from $\II$), then $\Sigma$ is product-free, that is, the partial type $\Sigma\cup \Sigma\cdot \Sigma\inv$ is inconsistent. 
    \item The group $G$ equals the connected component $\GO M$, where $\GO M$ denotes the intersection of all type-definable subgroups over $M$ of bounded index in $G$.
\end{itemize}
Moreover, if the second condition holds for some elementary submodel $M$, then it holds for every elementary submodel.
\end{teo}

The results from section \ref{S:Eq} will then be
applied in section \ref{S:IP} to study the asymptotics of IP sets as
well as of product-free sets. In particular, we obtain in Proposition
\ref{P:Nathanson} an elementary model-theoretic proof of an existing result
of Nathanson \cite[Theorem 6]{mN80}: every subset $A$ of the integers of
positive Banach density contains, for every two natural numbers $n$
and $m$, a subset of the form $F_1+ \cdots+ F_n + B$, for some finite
subsets $F_1,\ldots, F_n$ of $A$ of size $m$ and a subset $B$ of $A$
of positive Banach density.  In the last part of this note (Corollary \ref{C:quasi}),  we obtain a
non-quantitative finitary version of a result of Gowers \cite{wG08}
(in a weaker form). The infinite version of it (Corollary \ref{C:ultra}) 
resonates to a certain extent with previous work of Bergelson and Tao
\cite{BT14} on ultra quasirandom groups. We believe that this hints
to further connections, yet to be explored, between the techniques
used in the stabilizer theorem and the combinatorial approach to
complete $n$-amalgamation in terms of mixing and multiple
recurrence. The research we conducted here has been further developed in a subsequent article \cite{MPP24} for a wider class of groups (including pseudo-finite groups) to treat the case of complete $3$-amalgamation, which yields as a by-product a model-theoretic proof of Roth's theorem on the existence of arithmetic progressions of length 3 for subsets of positive density. 

\section{Keisler Measures and Ideals}\label{S:S1}

 Most of the material in this section can be found in \cite{eH12}.
 We  work
 inside a sufficiently saturated model $\mathbb U$ of a complete first-order
 theory (with infinite models) in a language $\LL$. All sets and
 tuples are
 taken inside $\mathbb U$.

A \emph{Keisler measure} $\mu$ is a finitely additive probability measure on
some boolean
algebra of definable subsets of the ambient model \cite{jK87}. Archetypal
examples are
measures $\mu_p$ with two possible values $0$ and
$1$ given by global
types $p$ in $S(\mathbb U)$, that is, for every definable set $X$,
\[ \mu_p(X)=1 \Leftrightarrow p \in [X],\]
where $[X]$ is the clopen space in $S(\mathbb U)$ determined by the definable
set $X$.

Given a Keisler measure $\mu$, the collection of definable sets of
measure zero forms an {\em ideal}, that is, it is closed under
subsets and finite unions. An ideal $\II$ of definable sets is
\emph{invariant} over the parameter set $A$ (or \emph{$A$-invariant}) if it is
invariant under the action
of $\mathrm{Aut}_A(\mathbb U)$. The ideal  $\II$ is
\emph{type-definable} over the parameter set $A$ if for every formula
$\varphi(x,y)$ there is a partial type $\Sigma_\varphi(y)$ over $A$ such that for
all $b$ in $\mathbb U^{|y|}$ the following equivalence holds:
\[ \Sigma_\varphi(b) \ \Leftrightarrow \ \varphi(x,b) \in \II.\]

A Keisler measure $\mu$ is said to be \emph{definable} over $A$
(see \cite[Definition 3.19]{sS17}) if for every
$\LL$-formula $\varphi(x,y)$ the following two properties hold:
\begin{itemize}
	\item the set of parameters $b$ in $\UU^{|y|}$ with $\varphi(x,b)$ measurable is definable by an $\LL_A$-formula $\rho(y)$.
	\item for every $\epsilon>0$, there is a partition of
	$\rho(\UU)$ into $\LL_A$-formulae $\rho_1(y),\ldots, \rho_m(y)$ such that for
	all pairs
	$(b, b')$  realizing $\rho_i(y)\land\rho_i(z)$, we have that
	\[ |\mu(\varphi(x,b)) -\mu(\varphi(x,b') )|<\epsilon. \]
\end{itemize}

Every finitely additive probability measure $\mu$ on all subsets admits an expansion of the original language $\LL$ in which it becomes definable without parameters, see for example \cite[Section 2.6]{eH12}. Namely, add a predicate $Q_{r,\varphi} (y)$ for each $r$ in $\mathbb Q \cap [0,1]$ and every formula  $\varphi(x,y)$ in $\LL$
such that $Q_{r,\varphi}(b)$ holds if and only if $\mu(\varphi(\mathbb U,b)) \le r$. These predicates $Q_{r,\varphi}$  give rise to new definable sets, which will also be measurable. Iterating this process countably many times and replacing the ambient model (if necessary), we obtain an expansion of the language $\LL$ such that the corresponding Keisler measure given by $\mu$ is definable.

Note that for an $A$-definable Keisler measure
$\mu$,  the ideal of sets of measure zero (called the
null-ideal of $\mu$) is type-definable
over $A$ and thus $A$-invariant.

\begin{remark}\label{R:wide_down_up}

	Let $\II$ be a type-definable ideal over an elementary submodel $M$. If
	$\varphi(x, b)$ does not belong to $\II$ for some $b$ in $\mathbb
	U$, then there is some $\mathcal L_M$-formula $\theta(y)$ such that
	$\varphi(x, c)$ does not belong to $\II$, whenever $c$ realizes
	$\theta(y)$.	In particular, there is some $m$ in $M$ such that
	$\varphi(x, m)$ does not belong to $\II$.
\end{remark}

\begin{definition}\label{D:S1}
An $\emptyset$-invariant ideal $\II$ of definable sets is \emph{indiscernibly
prime (or has the $S1$-property)}
if $\varphi(x, b_1)$ belongs to $\II$, whenever
$(b_i)_{i\in\mathbb N}$ is an indiscernible sequence such that the
conjunction
$\varphi(x, b_1) \land \varphi(x, b_2) $ lies in $\II$.
\end{definition}

 By invariance of $\II$,
we can replace the roles of $b_1$ and $b_2$ above by any two
$b_i$ and $b_j$
with $i\neq j$. The null-ideal of a definable Keisler measure is
indiscernibly
prime, since the measure of the whole space is bounded above by $1$.

A partial type is said to be \emph{wide} (with
respect to $\II$) if it contains no definable set in $\II$. The
definable set $X$ is wide if the partial type $X(x)$ is.  Every wide partial
type $\pi(x)$ over a
parameter set $A$
can be completed to a wide complete type over any arbitrary set
$B$
containing $A$, since the collection of formulae
\[ \pi(x)\cup\{\neg\varphi(x) \,|\, \varphi(x) \ \text{$\mathcal L_B$-formula}
\text{ in }
 \II \}\]
is finitely consistent.

\begin{remark}\label{R:wide_infinite}
		Let $\II$ be an indiscernibly prime $\emptyset$-invariant ideal containing
all the formulas of the form $x = a$ with $a$ in $\acl(\emptyset)$. Every wide definable set $X$ is infinite. In fact, the
		set $X$ does not fork over $\emptyset$ \cite[Lemma 2.9]{eH12}.
\end{remark}

In the presence of an ambient (type-)definable group $G$, it is meaningful to
require that the ideal $\II$ is invariant under left translation by elements
of the group. Note that natural action of $G$ on its definable sets induces an
action of $G$ on the set of (partial) types. We will identify a partial type
with the subset of its realizations in the structure $\UU$.

From now on, assume that the $\emptyset$-invariant
indiscernibly prime ideal $\II$ of $G$ is closed under left translations. Given
a partial type $\pi(x)$ over $A$, define the set
\[
\stt(\pi)=\{g\in G\,|\, \pi(x)\cup g\cdot \pi(x) \text{ is wide}\}
\]
\noindent and its \emph{stabilizer} $\stab(\pi)$  as the subgroup of $G$
generated by $\stt(\pi)$. Clearly, the
set $\stt(\pi)$ is not empty if and only if $\pi$ is wide itself. Since $\II$
is closed under left translations, the set $\stt(\pi)$ is closed
under taking
inverses in $G$. As $\stt(\pi)$ is contained in the set $\pi \cdot
\pi\inv$,  it follows that
\[ \stt(\pi)\subseteq \stab(\pi)=\bigcup\limits_{n\in \N}
\underbrace{\stt(\pi)\cdots \stt(\pi)}_{n} \subseteq \langle \pi \rangle_G,\]
\noindent whenever $\pi$ is wide, where $\langle \pi \rangle_G$ is the subgroup
of $G=G(\UU)$ generated by $\pi(\UU)$.

For our purposes, let us now state a summarized version Hrushovski's stabilizer theorem.

\begin{fact}(\cite[Theorem 3.5]{eH12} \& \cite[Theorem
2.12]{MOS18})\label{F:Hr}
Let $G$ be a definable group and $\II$ an $\emptyset$-invariant indiscernibly
prime ideal on the boolean algebra of all definable subsets of $G$
such that $\II$ is
translation-invariant under left multiplication. Given a wide complete
type $p$ over an elementary submodel $M$, the
subgroup $\stab(p)$ is  type-definable and wide with
\[\stab(p) = \stt(p)^2=(p\cdot p\inv)^2.\] Moreover, every wide type over $M$ in
$\stab(p)$ belongs to $\stt(p)$ and $\stab(p)$ equals $\GO M$, the intersection
of all subgroups of
bounded index in $G$ which are type-definable over $M$.
\end{fact}
The above result of Hrushovski generalises a well-known situation for stable
groups. Recall that a relation $R(x,y)$ is \emph{stable} with respect to
the
partition of the variables into the tuples $x$ and $y$ if there is no sequence
$(a_i, b_i)_{i\in\N}$ in $\UU$  such that $R(a_i,b_j)$
holds if and only if  $i\leq j$. A formula $\varphi(x,y)$ is stable if the
induced relation is.

Stable formulae are closed under boolean combinations. We will now introduce some notions and properties arising from the local treatment of stability done by Hrushovski and Pillay in \cite{HP94} (see also \cite{aP86}).  

Given a formula $\varphi(x,y)$, a set $X$ is
\emph{$\varphi$-definable} over a subset $A$ of parameters if it is
definable by a boolean combination of instances $\varphi(x,a)$ with $a$ in
$A$. By a \emph{$\varphi$-type} over a subset $A$ we mean a maximal finitely
consistent collection of instances of the form $\varphi(x,a)$ or
$\neg\varphi(x,a')$ for $a$ and $a'$ in $A$.

We say that a definable set $X$ is left-\emph{generic} if finitely
many left translates of $X$ cover $G$. A partial type is
left-\emph{generic} if it only contains formulae which are
left-generic.  If $\varphi(x,y)$ is stable, then left-generic
$\varphi$-types exist \cite[Lemma
5.16(i)]{HP94}. 

Whenever the formula $\varphi(x,y)$ is stable, every $\varphi$-type $p(x)$ over an elementary
submodel $M$ is \emph{definable}, that is, there is a formula $\theta(y)$ with
parameters over $M$ such that for all $m$ in $M$ \[ \varphi(x,m)
\in p \ \Longleftrightarrow \
\theta(m).\]   Furthermore, the definable set $\theta(y)$
above is unique and can be defined by a positive boolean
combination of instances $\varphi(a,y)$ with parameters in $M$ (cf. \cite[Lemma
5.4]{HP94}). We refer to this definable set as the \emph{$\varphi$-definition}
$(d_p\varphi)(y)$ of $p$.  Given a superset $B\supseteq M$ of $\UU$, there is a
unique $\varphi$-type over $B$ extending $p$ which is again definable over $M$,
namely \[ \{\varphi(x,b) \,|\, (d_p\varphi)(b)\}\cup \{\neg\varphi(x,b') \,|\,
\neg(d_p\varphi)(b')\}.\]
We refer to this type as the \emph{non-forking extension} $p\nf B(x)$
of $p(x)$ to
$B$. The \emph{global non-forking extension} of $p$ is the $\varphi$-type $p\nf
\UU$.

Whenever $\varphi(x,y)$ is {\em equivariant} (see \cite[Definition 5.13]{HP94}), that is, every left-
translate of an instance of $\varphi$ is again an instance of $\varphi$, the group $G$ acts naturally on the space of global $\varphi$-types by left multiplication, so given
a $\varphi$-type $p(x)$ over the elementary submodel $M$, it makes sense to consider its $\varphi$-\emph{stabilizer} to be the group-theoretic stabilizer of the global extension $p\nf
\UU$ with respect to this action. Note that 
\[\stab_\varphi(p)= \big\{g\in G\,|\, \forall u \big(
(d_p\varphi)(u)
\leftrightarrow (d_p\varphi)(u\cdot g) \big)   \big\} \]
 is clearly a definable subgroup of $G$ with parameters from $M$. However, this group need not be definable using instances of $\varphi$, cf. \cite[Remark 2.1]{MPPW19}. Conant, Pillay and Terry have provided alternative candidates for the $\varphi$-stabilizers in \cite[Theorem 2.3]{CPT20}.


\section{Equations and product-free sets}\label{S:Eq}

 Srour \cite{gS88} proposed a local version of noetherianity in terms of
 \emph{equations}. Recall that a relation $R(x;y)$ is an
\emph{equation} (for a given partition of the free variables into $x$
and $y$) if there is no sequence $(a_i,b_i)_{i\in\N}$ in $\UU$ such that\[
R(a_i,b_j)  \text { for all } i<j, \text { but  } \neg R(a_i,b_i)  \text { for all } i.\]

Equationality is a strengthening of stability, though not every stable relation
need be an equation, since stability is preserved under boolean combinations,
yet equationality is not. If the relation $R$ is $\emptyset$-invariant, we may
assume that the above sequence is indiscernible with respect to any possible
order type. In particular, equationality is symmetric in $x$ and $y$ for an $\emptyset$-invariant relation. Furthermore, it is easy to verify that an $\emptyset$-invariant relation $R$ is equational
if and only if, whenever the indiscernible sequence $(a_i)_{i\in\N}$ satisfies
that $R(a_i, b)$ holds for all $i>0$ with respect to some tuple $b$, then
$R(a_0, b)$ holds as well. Note that this last condition is a priori not symmetric in $x$ and $y$.

\begin{remark}\label{R:eq_nf}(\cite[Remark 2.4]{eH12} \& \cite[Remark
  3.6]{MPZ19}) Consider an $\emptyset$-invariant equational relation
  $R(x,y)$ such that $R(a,b)$ holds for some $a$ and $b$ such that
  $\tp(a/M,b)$ (or $\tp(b/M,a)$) does not divide over the elementary
  submodel $M$. For any $a'\equiv_M a$ and $b'\equiv_M b$, we have
  that $R(a',b')$.
\end{remark}

\begin{proof}
By symmetry of equationality, it suffices to consider the case that $R(a,b)$
holds for some $a$ and $b$ such that $\tp(a/M,b)$ does not divide over $M$. By
invariance, the above holds for all such $a_1$ and $b_1$ with
$(a_1,b_1)\equiv_M (a,b)$. Hence, we may assume that $b'=b$.

Let $a_0$ be $a'$ and choose for $i\geq 1$ a realization $a_{i}$ of
$\tp(a'/M)$ such that the type $\tp(a_{i}/M, b, a_0, \ldots, a_{i-1}) $ is a
coheir extension of its restriction to $M$. In particular, the type
$\tp(a_i/M,b)$ does not divide over $M$, for $i\ge 1$, so $R(a_i,b)$ holds for
$i\ge 1$, since the relation $R$ is stable \cite[Lemma 2.3]{eH12} (\cf
\cite[Lemma 3.3]{KP97}).

By construction, the sequence $(a_i)_{i\in\N}$ is $M$-indiscernible.
The equivalent characterization of equationality yields that $R(a_0, b)$, as
desired.
\end{proof}

Fix for the rest of this section an $\emptyset$-invariant
indiscernibly
prime ideal $\II$ on the boolean ring of
definable subsets contained in $(X\inv\cdot X)^2$, where $X$ is a
wide subset of a group $G$ definable over an elementary submodel $M$. In particular, the set-theoretic complements are taken within the ambient set $(X\inv\cdot X)^2$. Moreover, we assume that the ideal $\II$ is 
    translation-invariant under left multiplication by elements of the subgroup generated by $X$ (in the proofs of Corollary \ref{C:eq_st} and Lemma \ref{L:prod_free}, we will only need to consider products of at most four elements in $X\cup X\inv$): if $Z$ is a definable subset of $(X\inv\cdot X)^2$ and $g$ is an element of the subgroup generated by $X$ such that $g\cdot Z$ lies again in $(X\inv\cdot X)^2$, then $Z$ belongs to $\II$ if and only if $g\cdot Z$ does. 

\begin{notation}
Given two types $r$ and $s$ over $M$ containing the
formula $G(x)$, we consider
	the
	set
	\[ r\ast s= \{ b \cdot c \in G \,|\, b \models r, \ c\models s
	\text{
		and } \tp(c/M,b) \text{ does not divide over } M \}.
	\]
	In particular, we write $r  \ast r\inv$ for $r\ast \tp(a\inv\/M)$,
	where  $a\models r$.
\end{notation}

\begin{fact}\cite[Lemma 2.10]{eH12}\label{F:Hr_eq}
Whenever the $\emptyset$-invariant ideal $\II$ is indiscernibly
prime, for any two partial types $\Phi(x,y)$ and $\Psi(x,z)$
containing $X(x)$,
the relation \[ R_{\Phi,\Psi}(a,b) \ \Leftrightarrow \ \Phi(x,a) \cup
\Psi(x,b) \text{ is not wide}\]
is equational.

In particular, the relation $x\cdot y\in
\stt(p)$ is an equation, hence stable, for any type $p$ containing $X$. Thus, if $a\cdot b$
belongs  to $\stt(p)$ for some $a\cdot b$ in $r\ast s$, then
$r\ast s\subseteq \stt(p)$.
\end{fact}

The above fact  yields the following immediate consequence for
the equational relation $R_{\Phi,\Phi}$, which already appeared in \cite[Lemma 2.26(1)]{eH13}:

\begin{cor}\label{C:eq_wide}
If a partial  type $\Phi(x,a)$ is wide, so is
$\Phi(x,b)\cup\Phi(x,c)$,  whenever
$b$ and $c$ are realizations of $\tp(a/M)$ such that $\tp(c/M,b)$
does not divide over
$M$.  \qed
\end{cor}

\begin{cor}\label{C:eq_st}
Let $q$ be a wide type over $M$ containing the definable set $X$. Given any
(possibly non-wide) type $p$ over $M$ containing the definable
set $X$, the set $p \ast p\inv$ is contained in $ \stt(q)$. In
particular, if
$p$ is wide over $M$, then $p\ast p\inv \subseteq \stt(p)$, so $b
\cdot c\inv$ (and thus $c\cdot b\inv$) belongs to $\stt(q)$, whenever $b$ and $c$ realize $p$ and $\tp(c/M,b)$  does not
divide over
$M$.
\end{cor}

\begin{proof}
Since $q$ is wide, so is the type $\Phi(x,a)=a\inv\cdot q(x)$ for any $a$ in $X$. Fact
\ref{F:Hr_eq} and Corollary \ref{C:eq_wide} yield that 
$b\cdot c\inv$
belongs to $\stt(q)$ for any $b$ and $c$ realizing $p$ with
$\tp(c/M,b)$ non-dividing over $M$. Since $\stt(q)$ is closed under inverses, we deduce that  the product $c\cdot b\inv$ also lies in $\stt(q)$. We conclude that the set (of realizations of) $p \ast p\inv
\subseteq \stt(q)$, as desired. 
\end{proof}

Motivated by the second author's model-theoretic approach  in
\cite{dP19} to Gowers's solution \cite{wG08} to the asymptotics of
product-free sets \cite{BS85} in finite quasi-random groups, we say
that a partial type $\Sigma$ is \emph{product-free} if $\Sigma\cup (\Sigma\inv\cdot \Sigma)$
is inconsistent. Note that $\Sigma$ is not product-free if and only if
some realization $a$ of $\Sigma$ is of the
form $a=b\cdot c$, where $b$ and $c$ both realize $\Sigma$. If the complete type $p$ is not product-free, then every realization $a$ of $p$ is a product of two realizations of $p$.  A partial type is
not product-free if and only if it is the intersection of definable
sets which are not product-free. Therefore, if there exists a
product-free (partial) type over an elementary submodel, there is a
product-free (partial) type over every elementary submodel.

The next result can be easily proved along the lines of
Fact \ref{F:Hr}, as
shown in \cite[Lemma 3.3]{dP19}. For the sake of the
presentation, we will provide a direct
proof, which essentially follows the outline of  the proof
of \cite[Theorem 2.12]{MOS18}.

\begin{lemma}\label{L:prod_free}
  Given a wide type $p$ containing the definable set $X$ over $M$, the
  following conditions are equivalent:
\begin{enumerate}
	\item\label{E:stt} The type $p$ is contained in $\stt(p)$.
	\item\label{E:ast} The type $p$ is contained  in $(p\inv\ast p)$.
	\item\label{E:pf} The type $p$ is not product-free.
\end{enumerate}
\end{lemma}

\begin{proof}
  For $(\ref{E:stt})\Rightarrow (\ref{E:ast})$, suppose that a
  realization $b$ of $p$ is contained in $\stt(p)$, so $b\inv$ belongs
  to $\stt(p)$, and thus $p\cup b\inv\cdot p$ is wide with respect to
  $\II$.  In particular, there is a realization $a$ of
  $p\cup b\inv\cdot p$ which is wide over $M\cup\{b\}$. That is, the
  element $c=b \cdot a$ realizes $p$ and its type $\tp(c/M,b)$ does not divide over
  $M$, by Remark \ref{R:wide_infinite}. Hence
  $a=b\inv\cdot c$ belongs to $p\inv \ast p$, so every realization of
  $p$ does, as desired.

The implication $ (\ref{E:ast}) \Rightarrow (\ref{E:pf})$ is clear, for
$(p\inv\ast p)\subseteq (p\inv\cdot p)$.

For $(\ref{E:pf})\Rightarrow(\ref{E:stt})$, observe that $p\inv$ is not product-free whenever $p$ is not product-free. In particular, the complete type $p\inv$ is contained in $\stt(q)\cdot \stt(q)$: indeed, choose realizations $a$, $b$ and
$c$ of $p$ such that $a\inv=b\cdot c\inv$. Take now $d$
realizing $p$ wide over $M, a, b, c$ and notice that
\[
a\inv= b\cdot c\inv= (b\cdot d\inv) \cdot (d\cdot c\inv)
\stackrel{\ref{C:eq_st}}{\in} \stt(q)\cdot
\stt(q),
\]
by Corollary \ref{C:eq_st}. In particular, the type $p$ is contained in $\stt(q)\cdot \stt(q)$, for $\stt(q)$ is closed under inverses.

Choose now $b$ and $c$ realizing $p$ such that $\tp(c/M,b)$ is wide. Set $\eta=b\cdot c\inv$. Notice that $q=\tp(\eta/M)$ is wide, for it is the restriction of the wide type $\tp(\eta/M,b)$. Moreover, the type $q$ is contained in $\stt(p)$, by Corollary \ref{C:eq_st}.

\begin{claim}
	The set $\big(\stt(q)\cdot \stt(q)\big) \ast q$ is contained in
	$\stt(p)$, that is, given $\xi$ in $\stt(q)\cdot \stt(q)$ and
	$\eta_1$ realizing $q$ 	such that $\tp(\eta_1/M,\xi)$ does
	not divide over $M$, then $\xi \cdot \eta_1$ belongs to
	$\stt(p)$.
\end{claim}
\begin{claimproof}
	We first show that $\stt(q)\ast q$ is contained in $\stt(p)$. Given $\xi_1$ in
	$\stt(q)$ arbitrary, the type $q\cup\xi_1\cdot
	q$ is wide, so choose some realization $\eta_1$ of $q$ such
	that $\xi_1 \cdot \eta_1$
	realizes $q$ and $\tp(\eta_1/M,\xi_1)$ is wide, hence
	non-dividing over $M$ by Remark \ref{R:wide_infinite}. As remarked above, the type $q$ is
	contained in
	$\stt(p)$, so the element $\xi_1\cdot \eta_1$ of $\stt(q)\ast
	q$ belongs to $\stt(p)$. The relation $x\cdot y \in\stt(p)$ is stable, by Fact \ref{F:Hr_eq}, so we deduce that every element of $\stt(q)\ast q$ belongs to $\stt(p)$.   
 
	Now we prove that $\big(\stt(q)\cdot \stt(q)\big) \ast q\subseteq \stt(p)$.  Since the relation $x\cdot y \in\stt(p)$ is stable (see Fact \ref{F:Hr_eq}), we need only show that some product $\xi\cdot \eta_2$ belongs to $\stt(p)$ with $\tp(\eta_2/M, \xi)$ not dividing over $M$. Given a realization $\xi=\xi_1\cdot\xi_2$ with both
	$\xi_1$ and $\xi_2$ in $\stt(q)$, find as above some
	realization  $\eta_2$ of $q$
	such that the type $\tp(\eta_2/M,\xi_1,\xi_2)$ is wide and
	$\xi_2\cdot \eta_2$ realizes $q$, since $\xi_2$ belongs to
	$\stt(q)$.  In
	particular, the realization $\xi_2\cdot \eta_2$ of $q$ is wide over $M, \xi_1$, and hence $\tp(\xi_2\cdot \eta_2/M,\xi_1)$ does not divide over $M$, so the element $\xi_1\cdot
	(\xi_2\cdot
	\eta_2)=\xi\cdot\eta_2$ of $\stt(q)\ast
	q$ belongs to $\stt(p)$, by the first part of the proof of this claim. This yields the desired result. 
\end{claimproof}

To conclude the proof, recall that $p\inv$ is contained in $\stt(q)\cdot
\stt(q)$. Hence, the above realization $b\inv$ of $p\inv$
lies in
$\stt(q)\cdot
\stt(q)$. As the realization $\eta=b\cdot c\inv$ of $q$ is wide over
$M\cup\{b\}$,
we
deduce from the previous Claim that the element $b\inv\cdot\eta=c$ belongs
to $\stt(p)$. Since $c$ realizes $p$, this finishes the proof.
\end{proof}

To conclude this section, assume now that the group $G$ equals the
definable set $X$, and that both $G$ as well as the ideal $\II$ are
defined without parameters, thus the ideal $\II$ is indiscernibly
prime on the boolean algebra of all definable subsets of $G$. We
equip the quotient $G/\GO M$ with the logic topology, that is, a subset
in the quotient is closed if and only if its preimage in $G$ is
type-definable over $M$. The natural map from $G(M)$ to the compact
Hausdorff group $G/\GO M$ is the \emph{universal defi\-nable
  compactification} of the group $G(M)$ with respect to the elementary
structure $M$ \cite[Proposition 3.4]{GPP14}. We will establish an
equivalence between the triviality of the universal definable
compactification and the existence of wide product-free sets.

\begin{cor}\label{C:St-prod_free}
The following are equivalent:
\begin{enumerate}
		\item\label{E:sttp} Every wide type $p$ over an
		elementary submodel is contained in
		$\stt(p)$.
		\item\label{E:Prod_free} No wide type $p$ over an
		elementary submodel is product-free.
	\item\label{E:AllCtd} For every elementary submodel
	$N$, we have that $G=\GO N$.
\newcounter{saveenum}
\setcounter{saveenum}{\value{enumi}}
\end{enumerate}
Furthermore, if $\II$ is type-definable over $\emptyset$, then
each  of
the above conditions is equivalent to:
\begin{enumerate}
\setcounter{enumi}{\value{saveenum}}
	\item\label{E:SomeCtd} For some elementary
	submodel $M$, we have that $G=\GO M$.
\end{enumerate}
\end{cor}
\begin{proof}
	Lemma \ref{L:prod_free} yields that $(\ref{E:sttp})
	\Rightarrow(\ref{E:Prod_free}) $. For the implication
	$(\ref{E:Prod_free})\Rightarrow  (\ref{E:AllCtd})$, observe that any left
	coset
	of $\GO N$ contains a wide type, by \cite[Lemma 3.3]{eH12}. Therefore,
	there are no proper cosets of $\GO  N$, since a proper coset clearly defines a 
	product-free partial type. 	Hence,  all is left to show is the implication
	$(\ref{E:AllCtd})\Rightarrow
	(\ref{E:sttp})$. To see this, given a wide complete type
	$p$ over an elementary submodel $N$,  by the last remark in Fact \ref{F:Hr},
	the group $\stab(p)$ equals $\GO N$, which is $G$. Hence, the
	wide type $p$ is contained in $G=\stab(p)$,  so it must belong
	to $\stt(p)$, again by Fact \ref{F:Hr}, which yields the desired conclusion.

	Suppose now that $\II$ is type-definable over $\emptyset$.
        The implication
        $(\ref{E:AllCtd})\Rightarrow (\ref{E:SomeCtd})$ is immediate,
        so we need only show that
        $(\ref{E:SomeCtd})\Rightarrow (\ref{E:Prod_free})$. Suppose
        that $G=\GO M$ for some fixed elementary submodel $M$ and
        choose now a wide product-free type $p$ over an elementary
        submodel $N$. By compactness, there is a wide product-free
        $N$-definable set $A$ in $p$. Write $A=\varphi(x,n)$ for some
        tuple of parameters $n$ in $N$. Since $\II$ is type-definable,
        the Remark \ref{R:wide_down_up} yields that
        there is some $m$ in $M$ such that $\varphi(x,m)$ is wide and
        product-free. Any wide
        type $q$ over $M$ completing $\varphi(x,m)$ is
        product-free. The wide type $q$ lies in $G=\GO M$, so it must
        lie in $\stt(q)$, by Fact \ref{F:Hr}, contradicting Lemma
        \ref{L:prod_free}.
\end{proof}

A particular example of groups with trivial universal definable
compactifications are non-principal ultraproducts of the finite simple
groups $\mathrm{PSL}_2(\mathbb{F}_q)$, since such ultraproducts are
again simple groups. More generally, ultra quasirandom groups, studied
by Bergelson and Tao \cite[Definition 31]{BT14}, satisfy all of the
above conditions, as shown in \cite[Theorem 4.8]{dP19}, for ultra
quasirandom groups are particular ultraproducts of finite groups, and
hence can be seen as non-standard finite groups (see \cite[Definition
2.1]{aP17} \& \cite[Definition 2.4]{dP19}). In the following section,
we will explore the connection between product-free types and IP
sets. This connection was thoroughly studied in the work of Bergelson
and Tao \cite{BT14} for ultra quasirandom groups, motivated by a
quantitative version in the finite setting first exhibited for
quasirandom groups by Gowers \cite{wG08}.

\section{IP Sets and Measures}\label{S:IP}

The notions of lower, upper and Banach density (as stated in the
introduction) can be generalized to countable groups using F\o
lner sequences. A sequence
$(F_n)_{n\in\mathbb N}$ of finite sets of a countable group $G$ is
a (left) F\o lner sequence
 if $\lim\limits_{n\to \infty} |F_n| = 	\infty$ and such that
\[
\lim\limits_{n\to\infty} \frac{|F_n \cap g \cdot F_n|}{|F_n|} = 1,
\]
\noindent for all $g$ in $G$. Notice that a subsequence of a F\o
lner sequence is again F\o lner.

A  countable group is  \emph{amenable} if it can be
equipped with a F\o lner
sequence.  The
archetypal example of an amenable group  is $\mathbb Z$ with
the F\o lner sequence consisting of the bounded intervals
$F_n=[-n,n]$.
Every countable solvable (in particular, every abelian countable)
group is amenable.

 Given an amenable group $G$ with a distinguished F\o lner
 sequence $\mathcal  F=(F_n)_{n\in\mathbb N}$, we define the
 \emph{upper density} (with respect to $\mathcal F$) as
\[
\band F(A)= \limsup\limits_{n\to\infty}
\frac{|A\cap F_n|}{|F_n|} \text { for all } A\subseteq G.
\]
\noindent Note that $\band F(A)=\band F(g\cdot A)$ for all $g$
in $G$, since the sequence $\mathcal  F$ is F\o lner.  The
\emph{upper (left) Banach density} is then defined as
\[
d^*(A) = \sup \left\{ \band F(A) \ | \ \mathcal F \text{ is a (left)
F\o lner sequence in } G \right\}.
\]

\begin{remark}\label{R:BD_same}
  For $G=\mathbb Z$, this notion of upper Banach density and the one
  of the introduction are equivalent (see \cite[Remark 1.1 and Lemma 3.3]{BBF10}).
\end{remark}

A version of Furstenberg's correspondence principle for amenable
groups \cite[Theorem 5.8]{vB06}, or alternatively a clever use of
Hahn-Banach \cite[Proposition 4.19]{vB00}, produces from a set $A$ of
positive upper (Banach) density on an amenable group a left
translation invariant measure, which bounds from below the density of
any intersection of finitely many translates of $A$. For the sake of
the presentation, we will include a proof with no additional
input from  our  side.

\begin{fact}\label{F:Berg}
Let $\mathcal F=(F_n)_{n\in\mathbb N}$ be F\o lner sequence  in
a countable amenable group $G$ such that $\band F(A)>0$ for a
fixed subset $A$ of $G$. There exists a left translation invariant finitely
additive probability measure $\mu$ on $G$ with $\mu(A)=\bar
d_{\mathcal F}(A)$ such that for all $g_1,\ldots,g_r$  in $G$ we
have that
\[
\band F\left(g_1\cdot A \cap \ldots \cap g_r\cdot A \right) \geq
\mu\left(g_1\cdot A \cap \ldots \cap g_r\cdot A \right) .\]
\end{fact}
\begin{proof}
By definition of $\limsup$, there is a subsequence $\mathcal
F_0=(F_{n^0_k})_{k\in\N}$ of $\mathcal F$ such that \[\bar d_{\mathcal{F}_0}
(A)
= \band {\mathcal F} (A)=
\lim\limits_{k\to\infty} \frac{|A\cap F_{n^0_k}|}{|F_{n^0_k}|}.
\]
Notice that the sequence $\mathcal
F_0$ is again F\o lner. Since $G$ is countable, choose an enumeration
$(A_m)_{m\in \N}$
of all finite intersections of translates of $A$ by elements of $G$
with $A_0=A$ and find for each subset $A_m$ a F\o lner
subsequence $\mathcal F_{m}=(F_{n^{m}_{k}})_{k\in\mathbb N}$ of
$\mathcal F_{m-1}$ such that the limit
\[
\lim\limits_{k\to\infty} \frac{|A_m\cap F_{n^m_{k}}|}{|F_{n^m_{k}}|}
\]
exists. By a standard diagonal procedure, consider now $\mathcal F'$
with  $F'_m=F_{n^m_m}$. By  construction,  for every
$g_1,\ldots,g_r$ in $G$,  the limit
\[\lim\limits_{m\to\infty} \frac{\big| \bigcap\limits_{i=1}^r
(g_i\cdot A) \cap F'_m\big|}{|F'_m|} \]
exists, and bounds $\band F(g_1\cdot A \cap \ldots \cap g_r\cdot A)$
from below.

Choose now a non-principal ultrafilter $\mathcal U$ on $\mathbb N$ and
define the ultralimit with respect to $\mathcal U$
\[
\mu(B) = \lim\limits_{\mathcal U} \frac{\big| B \cap
	F'_m\big|}{|F'_m|}
    \] for each subset $B$ of $G$. By construction, the function $\mu$
    is a finitely additive probability measure on $G$ satisfying that
    $\mu(A) = \bar d_{\mathcal F}(A)$. Furthermore, for all
    $g_1,\ldots,g_r$ in $G$, we also have that
    $\band F(g_1\cdot A \cap \ldots \cap g_r\cdot A) \ge \mu(g_1\cdot
    A \cap \ldots \cap g_r\cdot A)$.  Finally, since the sequence
    $\mathcal F'$ is F\o lner, the above measure is left translation
    invariant.
\end{proof}
Every measure can be made definable over $\emptyset$ after possibly
expanding the language (see the discussion before the Remark \ref{R:wide_down_up}). Thus, we deduce from Fact \ref{F:Berg} the
following result:
\begin{cor}\label{C:BD_wide}
  For each set of positive upper (Banach) density $A$ of an amenable
  (countable) group $G$, there is an indiscernibly prime left
  translation invariant $\emptyset$-type-definable (hence, invariant)
  ideal $\II$ on the boolean algebra of all definable sub\-sets of $G$
  such that $A$ is wide with respect to $\II$. Furthermore, for any
  $g_1,\ldots,g_r$ in $G$, if the intersection
  $ \bigcap_{i=1}^r (g_i\cdot A) $ is wide, then it has
  positive upper (Banach) density.\qed
\end{cor}

As first proved in \cite[Theorem 2.6]{ACG19} (see also
\cite[Proposition 5.3]{gC19}), in the presence of
stability, a subset $A$ of positive upper Banach density in an
amenable discrete group has the productset property, that is, there are
infinite sets $B$ and $C$ such that $B\cdot C$ is contained in $A$.
We say that a set $A$ is stable if the equivariant relation
$A(y\cdot x)$ is stable. We would first like to clarify the relation between wide and generic sets under the assumption of stability (We thank the referee for pointing out that our original proof of Proposition \ref{P:BD_stable} could be reformulated in a simpler context).

\begin{remark}\label{R:stable_wide_gen}
Let $G$ be an infinite group. Left-generic definable subsets are wide for every left translation invariant  ideal $\II$ on the boolean algebra of all definable
subsets of $G$. Furthermore, if the definable subset $A$ is stable and wide for some indiscernibly prime left translation invariant $\emptyset$-invariant ideal $\II$ on the boolean algebra of all definable subsets of $G$, then $A$ is left-generic. Indeed, by \cite[Remark 5.17 (i)]{HP94}, it suffices to show that the formula $g\cdot A$ does not fork over $\emptyset$, which follows immediately from Remark \ref{R:wide_infinite}.
\end{remark}

We now observe that a slightly stronger result than \cite[Theorem 2.6]{ACG19}  
holds:

\begin{prop}\label{P:BD_stable}
Let $G$ be an infinite group  and $A$ a stable left-generic definable subset of $G$. Then there are an infinite definable subgroup $H$ of $G$ with $H\subseteq A\cdot A\inv$ as well as sequences $(b_n)_{n\in \mathbb N}$ in $H\subseteq A\cdot A\inv$ and $(c_n)_{n\in
\mathbb N}$ in $A$ such that $b_n\cdot c_m$ lies in $A$ for
all $n$ and $m$.
\end{prop}
In particular, stable wide sets have the productset property, by Remark \ref{R:stable_wide_gen}.
\begin{proof}
Let $\varphi(x,y)$ be the stable formula $A(y\cdot x)$, and choose by \cite[Claim in the proof of Lemma 5.16]{HP94} a generic
$\varphi$-type $p$, over the ambient model, containing
$A(x)=\varphi(x, \mathrm{id}_G)$. The $\varphi$-stabilizer $H\subset A\cdot A\inv$ of $p$ is
definable and  has finite index in $G$ by \cite[Lemma 5.16 (i)-(ii)]{HP94}. Therefore, the subgroup $H$ is infinite.

By stability of $\varphi$ and a standard Ramsey argument, we
need only show that there are infinite sequences $(b_n)_{n\in
\mathbb N}$ in $H\subseteq A\cdot A\inv$ and $(c_n)_{n\in 	\mathbb N}$ in $A$
such that $b_k\cdot c_m$ lies in $A$ if $k\leq m$. Suppose the
elements $b_0,\ldots, b_{n-1}$ and $c_0, \ldots,
c_{n-1}$ have already been constructed. Since $H$ is infinite,  choose $b_n$ in $H$
different from all $b_i$, with $0\leq i\leq n-1$. By definition of
the $\varphi$-stabilizer, the $\varphi$-type
$b_k\cdot p$ equals $p$ for all $0\leq k\leq n$. In particular, the $\varphi$-type $p$ contains the intersection $A\cap \bigcap_{k=0}^n b_k\inv\cdot A$, so this intersection cannot be finite. Choose therefore 
an element $c_n$ in $A\cap \bigcap_{k=0}^n b_k\inv\cdot A$ different
from all $c_i$, for $0\leq i\leq n-1$. By construction, the product $b_k\cdot c_n$ belongs to $A$ for $k\le n$, as desired.
\end{proof}
\begin{remark}\label{R:BD_stable}
An inspection of the above proof yields that it suffices to assume that the stable set $A$ contains a type $p$ whose $\varphi$-stabilizer is infinite. If the group is abelian, this is always the case if the infinite definable set $A$ is a boolean combination of weakly normal formulae by \cite[Lemma 2.4]{MPPW19}.

We do not know whether there are any further situations in which this weaker case can be applied. 
\end{remark}

By Corollary \ref{C:BD_wide}, a subset of positive upper Banach
density in an amenable group $G$ can be assumed to be definable and
wide with respect to a suitable ideal $\II$, after possibly expanding
the language. Consequently, we obtain the following immediate
observation and thank Gabe Conant for
pointing out a mistake in a previous version:

\begin{cor}\label{C:ACG_improved}
Every stable subset of positive upper Banach density in an infinite
	amenable group $G$ has the productset property for some infinite sets $B$
	and $C$, where $B$ lies in $A\cdot A\inv$ and
	$C$ is a subset of $A$.\qed
\end{cor}

Nathanson used a result of Khazdan in order to prove in
\cite[Theorem
6]{mN80} that a set $A$ of positive upper density in the integers
contains subsets of the form \[ F_1+\ldots+F_n+B,\] where each $F_i$
is finite of size at least $m$ and $B$ has positive upper density, for
every $n$ and $m$ in $\N$. By Corollary \ref{C:BD_wide}, we can extend
his result to any (possibly non-abelian) amenable group $G$.

\begin{prop}\label{P:Nathanson}
  Let $G$ be an amenable group. Given any two natural numbers $n$ and
  $m$, and a subset $A$ of positive upper density, there are finite
  subsets $F_1,\ldots, F_n$ of $A$ of size $m$ and a subset $B$ of $A$
  of positive upper density such that
	 \[ F_1\cdots F_n \cdot B \subseteq A.\]
\end{prop}
It will follow from the proof that $B$ can be chosen to be a finite intersection of left translates of $A$.
\begin{proof}
  It suffices to show the above result for $n=1$. Fix $m$ some natural
  number. By Corollary \ref{C:BD_wide}, equip $G$ with an
  $\emptyset$-type-definable left translation invariant indiscernibly
  prime ideal $\II$ on the boolean algebra of all definable subsets of
  $G$ (in an appropriate language) with $A$ wide and
  $\emptyset$-definable, such that for any $g_1,\ldots,g_r$ in $G$, if
  the intersection $ \bigcap\limits_{i=1}^r (g_i\cdot A) $ is wide,
  then it has positive upper (Banach) density.

  Choose now $g_1,\ldots, g_m$ in $A(\UU)$ starting an indiscernible
  sequence. Since $\II$ is indiscernibly prime, the set
  $\bigcap\limits_{i=1}^m (g_i\inv\cdot A(\UU) ) $ is wide. By the
  Remark \ref{R:wide_down_up}, we may find elements $h_1,\ldots, h_m$
  in $A$ with $B(\UU)=\bigcap\limits_{i=1}^m (h_i\inv\cdot A(\UU) ) $
  wide. Since the latter intersection is definable over $G$, we
  conclude that $B$ has positive upper density (as a subset of $G$).
  Furthermore, notice that $F_1\cdot B\subseteq A $ with
  $F_1=\{h_1,\ldots, h_m\}$, which finishes the proof.
\end{proof}

We will finish this section with a result on the IP property for wide
types in ultra quasirandom groups, along the lines of the work of
Bergelson and Tao \cite{BT14}. In the aforementioned work, the authors
take over work of Gowers \cite{wG08} in the finitary context, in order
to prove striking results on deterministic mixing of progressions and
IP systems. We will prove a much weaker result, concentrating solely
on inclusion in a given (wide) set, yet with an elementary proof from
the model-theoretic point of view.

\begin{theorem}\label{T:IP_wide}
Let $G$ be a group equipped with a left translation
invariant $\emptyset$-type-definable indiscernibly prime ideal
$\II$ on the
boolean algebra of all definable
subsets of $G$, and $A$ be a wide subset of $G$ definable over an
elementary substructure $M$ of $\UU$. If no wide type containing $A$ is
product-free, then $A(M)$ contains an IP set.
\end{theorem}

\begin{proof}
  Let $p$ be a wide type over $M$ containing $A(x)$. Since $p$ is not
  product-free, Lemma \ref{L:prod_free} implies that $p$ belongs to
  $\stt(p)$.  Choose some $a_0$ realizing $p$ in $\UU$, so
  $p\cup a_0\inv \cdot p$ is wide.  In particular, the subset
  $A\cap a_0\inv \cdot A$ is wide. As the ideal $\II$ is
  type-definable over $\emptyset$, we may find $a_0$ in $A(M)$ with
  $A_1=A\cap a_0\inv \cdot A$ wide, by Remark
  \ref{R:wide_down_up}. Every wide type (over $M$) containing the
  $M$-definable set $A_1$ must contain $A$, so no wide type containing
  $A_1$ is product-free either. We can iterate this process to find
  sequences $(a_i)_{i\in\N}$ of elements of $A(M)$ and
  $(A_k)_{k\in \N}$ of $M$-definable wide subsets of $A$ such that
  \[ a_{k+1} \text{ belongs to } A_{k+1}= A_k\cap a_k\inv\cdot A_k.\]
  It is now easy to show by induction on $m$ that a finite product
  $a_{i_1} \cdots a_{i_m}$ belongs to
  $A_{i_1}$ when $i_1<\ldots<i_m$.  Thus, the set $A(M)$ contains all finite ordered
  products of the sequence $(a_i)_{i\in\N}$, and hence is IP, as
  desired.
\end{proof}
The above result and Corollary \ref{C:St-prod_free} yield the
immediate consequence.

\begin{cor}\label{C:wide_IP}
Let $G$ be a group equipped with a left translation
invariant $\emptyset$-type-definable indiscernibly prime ideal
$\II$ on the
boolean algebra of all definable
subsets of $G$. If $G=\GO M$ for some elementary substructure $M$ of $\UU$,
then every wide subset of $G$ contains an IP set.\qed
\end{cor}

Recall the following notions in \cite[Definition before Theorem
4.6]{wG08} \& \cite[Definition 31]{BT14}. A finite group $G$ is
$d$\emph{-quasirandom} if it has no non-trivial representations over
$\mathbb C$ of dimension strictly less than $d$. An \emph{ultra
  quasirandom group} is an ultraproduct of finite groups $G_n$ with
respect to a non-principal ultrafilter $\mathcal U$ such that for
every $d$ in $\N$, the collection of $d$-quasirandom $G_n$ is
$\mathcal U$-large. It was proved in \cite[Theorem 4.8]{dP19} that
ultra quasirandom groups are precisely the ultraproducts of finite
groups satisfying the conditions in Corollary
\ref{C:St-prod_free}.  Therefore, Corollary \ref{C:wide_IP} yields a
weaker version of \cite[Lemma 40]{BT14}.

\begin{cor}\label{C:ultra}
  Any internal set $A$ in a ultra quasirandom group $G$ of positive
  (nonstandard) normalised counting measure contains an IP set.\qed
\end{cor}
A standard \L o\'s argument yields the following finitary version,
which is a non-quantitative weaker version of \cite[Theorem
5.3]{wG08}, with $A_F$ constant in Gowers's terminology.
\begin{cor}\label{C:quasi}
For every natural number $n$ and every $\epsilon>0$, there is some
$d=d(n,\epsilon)$ such that every finite subset $A$ in a $d$-quasirandom group
$G$ with $|A|\ge \epsilon |G|$ contains all possible products of a sequence of
length $n$.\qed
\end{cor}

\end{document}